\documentclass[12pt,reqno,oneside]{amsart}
\usepackage{amsmath}
\usepackage{amssymb}
\usepackage{amsthm}
\usepackage{graphics}
\usepackage{bm}
\usepackage[colorlinks]{hyperref}
\usepackage[capitalize,nameinlink,noabbrev]{cleveref}
\usepackage{mathrsfs}
\usepackage[normalem]{ulem}
\usepackage{color}
\usepackage[dvipsnames]{xcolor}
\usepackage{mathtools}
\usepackage{geometry}
\newtheorem{theorem}{Theorem}[section]
\newtheorem{definition}{Definition}[section]
\newtheorem{proposition}{Proposition}[section]
\newtheorem{lemma}{Lemma}[section]
\newtheorem{problem}{Problem}[section]
\newtheorem{corollary}{Corollary}[section]

\newtheorem{remark}{Remark}[section]

\newtheorem{dream theorem}{Dream theorem}[section]
\numberwithin{equation}{section}

\newcommand{\diag}{\operatorname{diag}}

\newcommand{\supp}{\operatorname{supp}}

\newcommand{\bq}{\mathbf{q}}

\newcommand{\bx}{\mathbf{x}}

\newcommand{\bp}{\mathbf{p}}

\newcommand{\bv}{\mathbf{v}}
\newcommand{\bw}{\mathbf{w}}
\newcommand{\br}{\mathbf{r}}

\newcommand{\by}{\mathbf{y}}

\newcommand{\Q}{\mathbb{Q}}

\newcommand{\Z}{\mathbb{Z}}
\newcommand{\R}{\mathbb{R}}
\newcommand{\N}{\mathbb{N}}
\newcommand{\ta}{\mathbf{\boldsymbol\theta}}

\usepackage{mathtools}
\usepackage{tikz-cd}
\usepackage{graphicx}

\begin{document}
\allowdisplaybreaks
\title {Winning of Inhomogeneous bad for curves}
\author{Shreyasi Datta}
\address{Department of Mathematics, University of York, Heslington, York, YO10
5DD, United Kingdom}
\email{shreyasi.datta@york.ac.uk, shreyasi1992datta@gmail.com}
\author{Liyang Shao}
\address{Department of Mathematics, University of California, Berkeley, CA 94720-3840, United States}
\email{liyang$_{-}$shao@berkeley.edu,  shaoliyang755@gmail.com}

\thanks{S.\ D.\ was partially supported by EPSRC grant EP/Y016769/1.}

\date{} 
\maketitle
\begin{abstract}
    We prove the absolute winning property of weighted  inhomogeneous badly approximable vectors on nondegenerate analytic curves. This answers a question by Beresnevich, Nesharim, and Yang. In particular, our result is an inhomogeneous version of the main result in [Duke Math J., 171(14), 2022] by Beresnevich, Nesharim, and Yang. 
    Also, the generality of the inhomogeneous part that we considered extends the previous result in [Adv. Math., 324:148–202, 2018]. Moreover, our results even contribute to classical results, namely establishing the inhomogeneous Schmidt's conjecture in arbitrary dimensions. 
    
\end{abstract}
\section{Introduction}

In this paper, we study weighted inhomogeneous badly approximable vectors for analytic curves in $\R^n.$ Let $\br=(r_1,\cdots,r_n)\in\R^n$ be such that $r_i\geq 0$ and $\sum_{i=1}^n r_i=1.$ Such $\br$ is called a \textit{weight} in $\R^n$. Let $\vert \cdot\vert_{\Z}$ denote the distance from the nearest integer. Given a weight $\br$, and $\ta=(\theta_i)_{i=1}^n$, where $\theta_i:\R^n\to \R,$ a map, we define the weighted inhomogeneous badly approximable vectors as follows:
\begin{equation}
\label{def: simul_inho}
    \mathbf{Bad}_{\ta}(\mathbf{r}):=\left\{\mathbf{x}\in\mathbb{R}^n:\liminf_{q\in\mathbb{Z}\setminus\{0\}, \vert q\vert \to\infty}\max_{1\leq i\leq n}\vert qx_i-\theta_i(\bx)\vert_{\mathbb{Z}}^{1/r_i}   \vert q\vert>0 \right\}.
\end{equation}
We denote $\bx=(x_i)\in\R^n$ where $x_i\in\R, 1\leq i\leq n,$ and we define $\vert x\vert^{1/0}:=0,$ for any $x\in (0,1).$
For $\ta(\bx)=\mathbf{0} ~\forall~ \bx\in\R^n$, i.e., in the homogeneous case, we denote $\mathbf{Bad}(\mathbf{r}):=\mathbf{Bad}_{\mathbf{0}}(\mathbf{r}).$ 
These sets defined in \cref{def: simul_inho} possess a rich structure,  that we will discuss soon, which makes them an interesting set in Diophantine approximation. Another reason these sets are interesting to a broader class of mathematicians is due to their connection with homogeneous dynamics, see \cite{Da}.

Jarn\'{i}k in \cite{Jarnik} showed that Hausdorff dimension of the bad set $\mathbf{Bad}(1)$ in $\R$ is $1$. In higher dimensions, when $\br=(1/n,\cdots,1/n)$, Schmidt showed that $\mathbf{Bad}(1/n,\cdots,1/n)$ is of full Hausdorff dimension, as a consequence of a stronger winning property with respect to a game he introduced in \cite{Schmidt}.

In \cite{Sc}, Schmidt conjectured that 
$$ \mathbf{Bad}(1/3,2/3)\cap \mathbf{Bad}(2/3,1/3)\neq \emptyset.$$ 
This conjecture was open for nearly 30 years, and gathered much attraction due to its connection with Littlewood's conjecture, another long-standing open problem, and homogeneous dynamics. In 2011, Schmidt's conjecture was settled assertively in \cite{BPV2011} in a stronger form. Moreover, the question of the winning property of $\mathbf{Bad}(\mathbf{r})$ with respect to the Schmidt game was first posed by Kleinbock in \cite{Kleinbock_Duke_98}. For dimension $2$, this question is addressed in \cite{An}, and in full generality for arbitrary dimensions, a stronger statement is recently proved in \cite{BNY21}.

In the fifties and sixties, Cassels, Davenport, and Schmidt first studied badly approximable vectors on some curves. Since then, understanding the behavior of the set of (weighted) badly approximable vectors inside manifolds and fractals has a very rich history and contains many open problems even to this date. As mentioned before, in a breakthrough \cite{BPV2011}, Schmidt's conjecture was proved; the heart of the proof requires showing that for any bad number $\alpha$, Hausdorff dimension of $L_{\alpha}\cap \textbf{Bad}(i,j)$ is $1$, where $L_{\alpha}$ is the line parallel to the $y$-axis and passing through $(\alpha,0).$ A higher-dimensional Schmidt conjecture is established in \cite{Beresnevich2015, Lei2019}. The papers \cite{BV11, Beresnevich2015, Lei2019} also answer Davenport's problem showing that $\mathbf{Bad}(\mathbf{r})$ has full Hausdorff dimension inside \textit{nondegenerate} manifolds in $\R^n.$ 

Other than the works we already mentioned, there are a series of works in the past few decades that establish Schmidt winning in different contexts; some of them are \cite{AGGL, GY, LDM, KL, BMS2017, LW, JLD, KW13, KW10, BFKRW, BFK, ET, Mos11, Jim1, Jim2, Fish09, KW2005, Da}.

A more desirable result of \textit{absolute winning} of $\mathbf{Bad}(\mathbf{r})$ inside nondegenerate analytic curves is recently proved by Beresnevich, Nesharim, and Yang in \cite{BNY22}. In \cite{McM2010}, this stronger notion of winning was first coined by McMullen, extending Schmidt's notion of winning in \cite{Schmidt}. We refer to \cite{McM2010, BFKRW, BHNS}, for the definition and several properties of the absolute winning sets.

In \cite{BNY22}, the authors ask if their result can be extended to the inhomogeneous setting. In this paper, we address that question by studying  $\mathbf{Bad}_{\ta}(\mathbf{r})$ inside nondegenerate analytic curves.

Our main theorem is as follows:

\begin{theorem}
\label{thm: main}
    Let $\br$ be a weight in $\R^n$.  Suppose $U\subset \R$ be an open interval in $\R$, and  $\varphi:U\to \R^n$ be an analytic map that is nondegenerate. Let $\ta=(\theta_i)_{i=1}^n$, where $\theta_i:\R^n \to \R$ are maps such that $\theta_i|_{\varphi(U)}$ are Lipschitz functions for $1\leq i \leq n$. Then 
    $\varphi^{-1}(\mathbf{Bad}_{\ta}(\mathbf{r}))$ is absolute winning on $U.$
\end{theorem}
 A map $\theta:A\subseteq \R^n\to \R$ is called Lipschitz on $A$ if there exists $d>0$ such that,
\begin{equation*}
    |\theta(\bx)-\theta(\by)|\leq d \Vert \bx-\by\Vert_{\infty}, \forall ~ \bx=(x_i),\by=(y_i)\in A,
\end{equation*} where $\Vert \cdot\Vert_\infty$ is the max norm.
Any continuously differentiable map on a closed bounded interval in $\R$ is always Lipschitz due to the mean value theorem. One can find the definition of a nondegenerate map in \cref{nondeg}.

Now we state some corollaries of \cref{thm: main}. Using the properties of absolute winning sets; see \cite[Lemma 1.2]{BNY22}, as an immediate consequence of \cref{thm: main}, we get the following.
\begin{corollary}\label{cor_curve}
    Let $n_i\in\N$, $i\geq 1$, $\{\ta^{i}_{j}\}_{j\geq 1}, \ta^{i}_j:\R^{n_i}\to\R^{n_i}$ be a sequence of Lipschitz maps, $W_i$ be a countable collection of weights in $\R^{n_i}$, and $U\subset \R,$ an open interval. Suppose for each $i\in\N$, $\varphi_i: U\to \R^{n_i}$ is an analytic nondegenerate map. Let $\mu$ be an Ahlfors regular measure on $U$ such that $U\cap\supp\mu\neq \emptyset.$
    Then 
    \begin{equation}
        \dim \bigcap_{i\geq 1}\bigcap_{\br\in W_i}\bigcap_{j\geq 1} \varphi_i^{-1}(\mathbf{Bad}_{\ta^{i}_{j}}(\mathbf{r}))\cap\supp\mu=\dim (\supp\mu).
    \end{equation}
    In case $\mu$ is the Lebesgue measure, 
    \begin{equation*}
        \dim \bigcap_{i\geq 1}\bigcap_{\br\in W_i}\bigcap_{j\geq 1} \varphi_i^{-1}(\mathbf{Bad}_{\ta^{i}_{j}}(\mathbf{r}))=1.
    \end{equation*}
\end{corollary}
The above corollary establishes the inhomogeneous version of Davenport's problem \cite{Dav64} for analytic curves. Furthermore, by using a fibering lemma as in \cite[\S 2.1]{Beresnevich2015}, and Marstrand’s slicing lemma \cite[Theorem 10.11]{Mat1995}, \cite[Lemma 7.12]{Fal2003}, we extend some parts of the previous corollary to any nondegenerate analytic manifold. The deduction of \cref{cor_manifold} from \cref{cor_curve} is exactly the same as in \cite[\S 2.1]{Beresnevich2015}.
\begin{corollary}\label{cor_manifold}
    Let $m,n,k\in \N$, $1\leq i\leq k,$ $\{\ta^{i}_{j}\}_{j\geq 1}, \ta^{i}_j:\R^{n}\to\R^{n}$ be a sequence of Lipschitz maps, $W$ be a countable set of weights in $\R^{n}$, and $U\subset \R^m,$ be an open ball. Suppose for each $1\leq i \leq k$, $\mathbf{\varphi}_i: U\to \R^{n}$ is an analytic nondegenerate map. 
    Then  
    \begin{equation*}
        \dim \bigcap_{i=1}^k\bigcap_{\br\in W}\bigcap_{j\geq 1} \varphi_i^{-1}(\mathbf{Bad}_{\ta^{i}_{j}}(\mathbf{r}))=m.
    \end{equation*}
\end{corollary}
\begin{remark}
     As pointed out before, Corollary \ref{cor_curve} implies Corollary \ref{cor_manifold}. For curves, Corollary \ref{cor_curve} is stronger than Corollary \ref{cor_manifold}. Because in the former case we consider infinite intersection of curves embedded in different dimensions, whereas in the later case, we can take only finite intersection of manifolds embedded in the same dimension.
\end{remark}
By \cite[Lemma 1.2 (iv) ]{BNY22}, (and \cite{BHNS}) \cref{thm: main} implies that, for any Ahlfors regular measure $\mu,$ with $U\cap \supp\mu\neq\emptyset$, we must have $\varphi^{-1}(\mathbf{Bad}_{\ta}(\mathbf{r}))\bigcap\supp\mu\neq\emptyset.$ Surprisingly, by \cite[Proposition 2.18, Corollary 4.2]{BHNS} (\cite[Lemmata 1.2-1.4]{BNY22}), the last implication is sufficient to prove absolute winning. Hence, we prove the following theorem that is equivalent to \cref{thm: main}.

\begin{theorem}
\label{thm: twin main}
Let $\ta, \mathbf{r},U,\varphi$, be as in \cref{thm: main}. Suppose $\mu$ be a $(C,\alpha)$-Ahlfors regular measure such that $U\cap \supp{\mu}\neq \emptyset.$ 
Then 
\begin{equation}
    \varphi^{-1}(\mathbf{Bad}_{\ta}(\mathbf{r}))\bigcap\supp\mu\neq\emptyset.
\end{equation}
\end{theorem}
The relevant definition of \textit{Ahlfors regular measure} can be found in \S ~\cref{preli}.
Next, we point out how our main theorem relates to previous results in this direction. \noindent
\subsection*{Remarks}
\begin{enumerate}
\item \cref{thm: main} represents complete inhomogeneous version of \cite[Theorem 1.1]{BNY22}.
\item For $n=2$, \cref{thm: main} extends \cite[Theorem 1.1]{ABV} from $\ta$ being a constant to any Lipschitz function.
\item \cref{cor_manifold} gives the inhomogeneous analogue of \cite[Theorem 1]{Beresnevich2015}, and in particular, answers Davenport's questions \cite[p. 52]{Dav64} in much more generality.
\item\cref{cor_manifold}, for $n=m\geq 2$, $\varphi_i(\bx)=\bx ~\forall~ \bx\in\R^n, 1\leq i\leq k$ 
is also new, giving a contribution to the classical setup, namely solving the inhomogeneous analogue of Schmidt’s conjecture in arbitrary dimensions.
\end{enumerate}

The basic idea of the proof of \cref{thm: twin main} is inspired by the works \cite{ABV,BV13}. Our approach is to build on the construction of Cantor sets in \cite{BNY22}. Specifically, 
we construct a new Cantor set inside the Cantor set constructed in \cite{BNY22}, such that the new Cantor set is also inside the inhomogeneous bad sets; see \cite[Section 3]{BV13} for the philosophy behind the strategy and \S ~\cref{proof-of-main} for details. One of the main challenges in this paper comes from the fact that we consider the inhomogeneous part to be a function. This is a more general setup, compared to \cite{ABV, BV13}, 
which gives rise to the difficulties that the desired Cantor set needs to avoid certain sets that are not intervals. To overcome this, vaguely speaking, we approximate the sets that we want to avoid by certain intervals. This also requires us to categorize rationals in a more complex manner than in \cite{ABV}.


\section{Preliminaries}\label{preli}
Let us first recall the definition of Ahlfors regular measures from \cite[\S 1]{BNY22}.

\begin{definition}
\label{ahlfors}
    A Borel measure $\mu$ on $\R^d$ is called $(C,\alpha)$-Ahlfors regular if there exists $\rho_0>0$ such that for any ball $B(\bx,\rho)\subset\R^d$ where $\bx\in\supp\mu$ and $\rho\leq\rho_0$, we have $$
    C^{-1}\rho^{\alpha}\leq\mu(B(\bx,\rho))\leq C\rho^{\alpha}.
    $$
\end{definition}

Let us recall the definition of a nondegenerate map. 
\begin{definition}
    \label{nondeg}
    An analytic map $\varphi=(\varphi_1,\cdots,\varphi_n):U\to \R^n$, defined on an open interval $U\subset \R$, is nondegenerate if $\varphi_1,\cdots,\varphi_n$ together with the constant function 1 are linearly independent over $\R.$
\end{definition}
The above definition is restrictive to analytic maps. In general, the definition of nondegenerate map with respect to a measure can be found in \cite{KM}.

For any $g\in \mathrm{SL}_{n+1}(\R)$, $g\Z^{n+1}$ is a lattice in $\R^{n+1}$ with covolume $1$. Thus, by the equivalence map $g\mathrm{SL}_{n+1}(\Z)$ to $g\Z^{n+1}$, we associate $X_{n+1}$ with the space of all lattices with covolume $1$ (i.e., unimodular). For every $\varepsilon>0$, let us define 
$$
K_{\varepsilon}:=\{\Lambda\in X_{n+1}:=\mathrm{SL_{n+1}}(\R)/\mathrm{SL_{n+1}}(\Z)~|~ \inf_{\bx\in \Lambda\setminus 0}\Vert \bx\Vert\geq \varepsilon\}. $$ Mahler's compactness criterion ensures that these are compact sets in $X_{n+1}.$ Here and elsewhere $\Vert \cdot\Vert$ is the Euclidean norm.

Next, we recall the definition of generalized Cantor set, following \cite[\S 2.2]{BNY22}.
\begin{definition}
    Given $R\in\N$, $I\subset\R$ a closed interval, the $R$-partition of $I$ is the collection of closed intervals obtained by dividing $I$ to $R$ closed subintervals of the same length $R^{-1}|I|$, denoted as \textbf{Par}$_R(I)$. Here, $\vert I\vert$ is the Lebesgue measure of the interval $I.$

    Also, given $\mathcal{J}$ a collection of closed intervals,$$\mathbf{Par}_R(\mathcal{J}):=\cup_{I\in\mathcal{J}}\mathbf{Par}_R(I).$$
\end{definition}

To define a generalized Cantor set, we first let $\mathcal{J}_0=\{I_0\}$, where $I_0$ is a closed interval. Then inductively $\mathcal{I}_{q+1}=\mathbf{Par}_R(\mathcal{J}_q)$. In addition, we remove a subcollection $\hat{\mathcal{J}}_q$ from $\mathcal{I}_{q+1}$. In other words, we let $\mathcal{J}_{q+1}=\mathcal{I}_{q+1}\setminus\hat{\mathcal{J}}_q$. In particular,
\begin{equation}
    \label{length}
    \forall I\in\mathcal{J}_q,|I|=|I_0|R^{-q}.
\end{equation}
We define the generalized Cantor set,
\begin{equation}\label{K_infty}\mathcal{K}_{\infty}:=\bigcap_{q\in\N\cup\{0\}}\bigcup_{I\in\mathcal{J}_q}I.\end{equation}

Next, we recall the following theorem that we will use to get the non-emptiness of a generalized Cantor set.

\begin{theorem}{\cite[Theorem 3]{BV11multi}}
    \label{nemp}
   $\forall q\in\N\cup\{0\}$, suppose $\hat{\mathcal{J}}_q$ can be written as \begin{equation}\hat{\mathcal{J}}_q=\cup_{p=0}^q\hat{\mathcal{J}}_{p,q}\label{jq}.\end{equation} Let us denote $h_{p,q}:=\max_{J\in\mathcal{J}_p}\#\{I\in\hat{\mathcal{J}}_{p,q}:I\subset J\}$. If \begin{equation}
        \label{formula}
        t_0:= R- h_{0,0}>0, t_q:=R-h_{q,q}-\sum_{j=1}^q\frac{h_{q-j,q}}{\prod_{i=1}^jt_{q-i}}>0, ~\forall~ q\geq 1,
    \end{equation}
    then $\mathcal{K}_{\infty}\neq\emptyset.$
\end{theorem}

We also recall the following transference lemma from \cite{Mahler1939}, which we state in the form as it appeared in \cite[Lemma 11]{Beresnevich2015}.
\begin{lemma}
\label{transfer}
    Let $T_0,\cdots, T_n>0$, $L_0,L_1,\cdots,L_n$ be a system of linear forms in variables $u_0,u_1,\cdots,u_n$ with real coefficients and determinants $d\neq 0$, and let $L_0',L_1',\cdots,L_n'$ be the transposed system of linear forms in variables $v_0,v_1,\cdots,v_n$, so that $\sum_{i=0}^nL_iL_i'=\sum_{i=0}^nu_iv_i$. Let $\iota^n=\frac{\prod_{i=0}^nT_i}{|d|}.$ Suppose that there exists $(u_0,u_1,\cdots,u_n)\in\Z^{n+1}\setminus \{\mathbf{0}\}$ such that
    \begin{equation}
        \label{system1}
        |L_i(u_0,u_1,\cdots,u_n)|\leq T_i,~~\forall~~ 0\leq i\leq n.
    \end{equation}
    Then there exists a non-zero integer point $(v_0,v_1,\cdots,v_n)$ such that
    \begin{equation}
        \label{system2}
        |L_0'(v_0,v_1,\cdots,v_n)|\leq n\iota/T_0\textit{ and }|L'_i(v_0,v_1,\cdots,v_n)|\leq \iota/T_i,~~\forall~~ 1\leq i\leq n.
    \end{equation}
\end{lemma}
\subsection{Notations and auxiliary results}\label{subsection: Notation}
Without loss of generality, we assume,
\begin{equation}\label{weights chain}
r_1\geq r_2\geq \cdots\geq r_n>0,
\end{equation} where $\mathbf{r}$ is the weight in $\R^n$ that we considered in \cref{thm: main}.  Without loss of generality, we can take $r_n>0$ as otherwise, $\textbf{Bad}_{\hat\ta}{(\hat\br)}\times \R= \textbf{Bad}_{\ta}(\br),$ where $\hat\br=(r_1,\cdots,r_{n-1}),$ and $\hat\ta=(\theta_i)_{i=1}^{n-1}.$ 

We take $\varphi$ be a nondegenerate analytic map with domain $U$ as in Definition \ref{nondeg}. Let us take $\mu$ to be $(C,\alpha)$-Ahlfors regular measure on $\R$ such that $U\cap\supp\mu\neq \emptyset$ with $\rho_0$ be defined as \cref{ahlfors}. Let us take $x_0\in \supp \mu\cap U$ such that $\varphi'(x_0)\neq 0$. This is possible because $\varphi'$ only vanishes at countably many points whereas $U\cap \supp\mu$ is uncountable. Following \cite[\S 2.1]{BNY22}, since $\varphi_1'(x_0)\neq 0$, by taking $I_0$ small enough, we can guarantee that $\vert \varphi'_1\vert$ is bounded below and above by some positive numbers. Therefore, by making change of variables and shrinking $I_0$ if necessary, we can assume $I_0$ is an interval centered at $x_0$ such that (similar to \cite[Equation 2.2]{BNY22}),
\begin{equation}
    \label{Change of variable}
    \varphi(x)=(\varphi_1(x),\varphi_2(x),\cdots,\varphi_n(x)), \varphi_1(x)=x, \forall x\in 3^{n+1}I_0,
\end{equation} with \begin{equation}\label{U containing I_0}3^{n+1}I_0\subset U,\end{equation}
and \begin{equation}\label{eqn: lengthI_0}3|I_0|\leq\min\{\rho_0, 1\}.\end{equation}
Let $\ta$ be as in \cref{thm: main}. Define
$$\varphi^{-1}(\mathbf{Bad}_{\ta}(\mathbf{r})):=\{x\in\mathbb{R}:\liminf_{q\in\mathbb{Z}\setminus\{0\}, |q|\to\infty}\max_{1\leq i\leq n}\vert q\varphi_i(x)-\theta_i(\varphi(x))\vert_{\mathbb{Z}}^{1/r_i}   \vert q\vert>0 \}.$$
Let $\theta^{\varphi}_i:U\to\R$, be defined as  $\theta^{\varphi}_i(x):=\theta_i(\varphi(x)).$
Since each $\theta_i$ is Lipschitz on $\varphi(U)$, and $\varphi$ is analytic, we have $\theta^{\varphi}_i$ to be Lipschitz on $3^{n+1}I_0$.
We will drop the suffix $\varphi$ here onwards and write $\theta_i$ for clarity of notation.
So for every $1\leq i \leq n$ there exists $d_i>0$ such that,
\begin{equation}\label{eqn:d_i}
    |\theta_i(x)-\theta_i(y)|\leq d_i \vert x-y\vert, \forall ~ x,y\in 3^{n+1}I_0.
\end{equation}

Let us recall the definitions of the following matrices in $\mathrm{SL}_{n+1}(\R)$ as in \cite{BNY22}:
\begin{equation}
    \begin{aligned} & a(t):=\diag\{e^t,e^{-r_1t},\cdots, e^{-r_nt}\},b(t):=\diag\{e^{-t/n},e^{t},e^{-t/n}, \cdots,e^{-t/n}\},\\
    & u(\mathbf{x}):=\begin{pmatrix}
        1&\mathbf{x}\\
        0&\mathrm{I}_n
    \end{pmatrix}, \bx\in \R^n, \text{ where } \mathrm{I}_n \text{ is the identity matrix of size  } n\times n.\end{aligned}
\end{equation}
For any $\by\in\R^{n-1}$, and $x\in\R,$ let us define the following matrix in $\mathrm{SL}_{n+1}(\R),$
\begin{equation}
    \label{inho3}
    u_1(\by):=\begin{pmatrix}
        1&0&0& 0& \cdots&0\\
         0&1&y_2& y_3& \cdots&y_n\\
         \mathbf{0}&\mathbf{0} & & &\mathrm{I}_{n-1}\\
    \end{pmatrix}, ~~z(x):= u_1(\hat\varphi'(x)),
\end{equation}
where $\hat\varphi(x):=(\varphi_2(x),\cdots, \varphi_n(x)).$ By \S~ \cref{preli}, the above matrices associate to elements in $X_{n+1}$ by the equivalence map.
Let $R>0$ be a fixed large number and $\beta,\beta'>1$ be such that
\begin{equation}
\label{beta}
    e^{(1+r_1)\beta}=R=e^{(1+1/n)\beta'}.
\end{equation}
Also, let us fix $\epsilon>0$ such that \begin{equation}\label{def: epsilon}
0<\epsilon<\frac{r_n}{4n}.\end{equation}

Note that we are working with a fixed $\mu$ as in Theorem \ref{thm: twin main}, and an interval $I_0$ was chosen depending on $\mu$, as in the beginning of this subsection.  In \cite{BNY22}, a non-empty generalized Cantor set $\mathcal{K}_{\infty} $ determined by $\hat{\mathcal{J}}_q=\cup_{p=0}^q\hat{\mathcal{J}}_{p,q}$ has already been constructed within $\varphi^{-1}(\mathbf{Bad}(\mathbf{r}))\cap\supp\mu$; see \cite[Proposition 2.6]{BNY22}. We recall the construction. Suppose that
\begin{equation}
    \label{jqq}
    \hat{\mathcal{J}}_{q,q}:=\left\{I\in\mathcal{I}_{q+1},\mu(I)<(3C)^{-1}|I|^{\alpha}\right\}.
\end{equation}
Note $C, \alpha$ are the quantities associated to Ahlfors regular property of the measure $\mu$. Next, for $\forall q\geq 1$ define $\hat{\mathcal{J}}_{p,q}=\emptyset$ if $0<p\leq q/2$ or $0<p<q$ with $p\not\equiv q\pmod 4$, define $\hat{\mathcal{J}}_{0,q}$ to be the collection of $I\in\mathcal{I}_{q+1}\setminus\hat{\mathcal{J}}_{q,q}$ such that there exists $l\in\Z$ with $\max(1,q/8)\leq l\leq q/4$ satisfying 
\begin{equation}
    \label{jpq}
b(\beta'l)a(\beta(q+1))z(x)u(\varphi(x))\Z^{n+1}\notin K_{e^{-\epsilon\beta l}}\textit{ for some }x\in I,
\end{equation}
and finally if $q/2<p<q$ and $p=q-4l$ for some $l\in\Z$ define 
\begin{equation*}
    \hat{\mathcal{J}}_{p,q}:=\left\{I\in\mathcal{I}_{q+1}\setminus\left(\hat{\mathcal{J}}_{q,q}\cup\bigcup_{0\leq p'<p}\hat{\mathcal{J}}_{p',q}\right):\cref{jpq}\textit{ holds}\right\}
\end{equation*}
    
Then $\hat{\mathcal{J}}_{q}$ is as defined in \cref{jq}, $\mathcal{J}_{q+1}=\mathcal{I}_{q+1}\setminus\hat{\mathcal{J}}_{q}$ and $\mathcal{K}_{\infty}$ is as defined in \cref{K_infty}.

Then let us recall the following lemma that was proved in \cite[Equation 2.35]{BNY22}.

\begin{lemma}
\label{234}
Given $q>8$, then $\forall I\in\mathcal{J}_{q}$, we have 

\begin{equation}\label{eqn: Dani reversing b}
\forall x\in I,a(\beta q)u(\varphi(x))\Z^{n+1}\in u_1(O(1))^{-1}b(-\beta')K_{e^{-\epsilon\beta}}.
\end{equation}
\end{lemma}

In the above lemma, $O(1)$ is a vector in $\R^{n-1}$ whose norm is $\ll 1.$ 

We denote $h_{p,q}:=\max_{J\in\mathcal{J}_p}\#\{I\in\hat{\mathcal{J}}_{p,q}:I\subset J\},0\leq p\leq q$. The following two propositions provide upper bounds for $h_{p,q}$.

\begin{proposition}{\cite[Proposition 2.7]{BNY22}}
    Let $C,\alpha$ be as defined in \cref{thm: twin main} and $R^\alpha\geq 21C^2$, for $\forall q\in\N\cup\{0\},$
    \begin{equation}
        \label{hqq}
        h_{q,q}\leq R-(4C)^{-2}R^{\alpha}.
    \end{equation}
\end{proposition}

\begin{proposition}\cite[Propositions 3.3 and 3.4]{BNY22}\label{hpq}
    There exist constants $R_0\geq1,C_0>0$ and $\eta_0>0$ such that if $R\geq R_0$, then for any $p,q\in\N\cup\{0\}$ with $p<q$, we have 
\begin{equation}
    h_{p,q}\leq C_0R^{\alpha(1-\eta_0)(q-p+1)}.
    \end{equation}
\end{proposition}
\section{Proof of the main theorem}\label{proof-of-main}
\begin{lemma}
    \label{bounded}
  There exists $\xi>0$ such that $\forall ~q\in\N$, ~$\forall I\in\mathcal{J}_{q},\forall x\in I$,
    \begin{equation}
        \{a(\beta t)u(\varphi(x))\Z^{n+1}, 0<t\leq q\}\subset K_{\xi}.
    \end{equation}
\end{lemma}
\begin{proof}
    Note that the right-hand side of \cref{eqn: Dani reversing b} is contained in a bounded subset of $X_{n+1}$, independent of $q$ and $x$. Therefore, by \cref{234},  $\{a(\beta q_0)u(\varphi(x))\Z^{n+1}:8<q_0\leq q,q_0\in\N\}$ is in a bounded set in $X_{n+1}$, that is independent of $q$ and $x$. Moreover, when $0<q_0\leq8,$ $a(\beta q_0)u(\varphi(x))\Z^{n+1}=a(-\beta(9-q_0))a(9\beta)u(\varphi(x))\Z^{n+1}$. Since $a(9\beta)u(\varphi(x))\Z^{n+1}$ is bounded independent of $x$ by previous analysis and $1\leq 9-q_0<9$ ensures that $a(-\beta(9-q_0))$ is also bounded, we have  $\{a(\beta q_0)u(\varphi(x))\Z^{n+1},0<q_0\leq 8\}$ is in a bounded set independent of $q$ and $x$. 
    Thus, $\{a(\beta q_0)u(\varphi(x))\Z^{n+1}:q_0\leq q,q_0\in\N\}$ is in a bounded set independent of $q$ and $x$.

    If $q_0-1<t\leq q_0$ with $q_0\in\N$, then $$a(\beta t)u(\varphi(x))\Z^{n+1}= a(-\beta(q_0-t))a(\beta q_0)u(\varphi(x))\Z^{n+1}.$$ Since $0\leq q_0-t<1$, $\{a(\beta t)u(\varphi(x))\Z^{n+1}, 0<t\leq q, x\in I, I\in \mathcal{J}_q, q\in\N\}$ is bounded in $X_{n+1}$ independent of $q$ and $x$. So, by Mahler's compactness criterion, the lemma follows.

\end{proof}
Let $\xi$ be as in \cref{bounded}, and without loss of generality, we assume \begin{equation}
    \xi<\frac{n+1}{e}\label{xi}.
\end{equation} We define the following quantities:

\begin{equation}\label{def: lambda_1}\lambda_1:=\left\lceil\ln{\left(\frac{n+1}{\xi}\right)^{\frac{1}{\beta r_n}}}\right\rceil,\end{equation}
\begin{equation}
    k_1:=\frac{\xi}{n+1}e^{-\beta\lambda_1}.
    \label{def: k_1}
\end{equation}
Note that $\beta$ depends on $R$, and therefore, in the above definitions $\lambda_1\in \N$ and $k_1<1$ are dependent on $R.$
From \cref{def: lambda_1}, we have \begin{equation}\label{eqn: defn of lambda_1 gives}
e^{-\beta\lambda_1 r_i}\leq\frac{\xi}{n+1}, 1\leq i\leq n.
\end{equation}
Also, \begin{equation}\label{upper bound on k_1r_1}
    k_1^{1+r_1}<R^{-\lambda_1},
\end{equation}
since by \cref{xi}, $\frac{n+1}{\xi}\stackrel{\eqref{def: k_1}}{=}k_1^{-1}e^{-\beta\lambda_1}>1\implies k_1<e^{-\beta\lambda_1}\stackrel{\eqref{beta}}{\implies} k_1^{1+r_1}<R^{-\lambda_1}.$

We then have the following corollary. 
\begin{corollary}
\label{dual}
     $\forall q\in\N, q>\lambda_1, ~\forall ~1\leq H<e^{\beta(q-\lambda_1)},~\forall I\in\mathcal{J}_q,\forall x\in I,$  there are no non-zero integer solutions $(a_0,a_1,\cdots,a_n)$ to \begin{equation}|a_0+\sum_{i=1}^na_i\varphi_i(x)|\leq k_1H^{-1},|a_i|\leq H^{r_i},~ 1\leq i \leq n.\label{dual system}\end{equation}
\end{corollary}

\begin{proof}
    We prove this by contradiction. Let $(a_0,a_1,\cdots,a_n)\in\Z^{n+1}\setminus \{\mathbf{0}\}$ be a solution to \cref{dual system} with a positive number \begin{equation}1\leq H<e^{\beta(q-\lambda_1)}\label{H},\end{equation} and $x\in I$ for some $I\in\mathcal{J}_q$. Then, we let 
    \begin{equation}
        \label{t_0}
        t_0:=\beta\lambda_1+\ln H,
    \end{equation}
    Thus by \cref{H} and $\lambda_1>0$,\begin{equation}\label{t_0'}0< t_0<\beta q.\end{equation}
    
    Also, let
    \begin{equation}
        \label{w}
        \bw:=(w_i)_{i=0}^n=a(t_0)u(\varphi(x))\cdot (a_0,a_1,\cdots,a_n)^T.
    \end{equation}

    Thus, we have $\vert w_0\vert =e^{\beta\lambda_1}H|a_0+\sum_{i=1}^na_i\varphi_i(x)|\stackrel{\eqref{dual system},\eqref{def: k_1}}{\leq}\frac{\xi}{n+1}$ and $\forall 1\leq i\leq n$, $\vert w_i\vert =e^{-\beta\lambda_1r_i}H^{-r_i}|a_i|\stackrel{\eqref{dual system}, \eqref{eqn: defn of lambda_1 gives}}{\leq}\frac{\xi}{n+1}$, thus $0\neq \Vert \bw\Vert=\sqrt{\sum_{i=0}^nw_i^2}<\xi$.

    This implies that there exists $x\in I,$ for some $I\in \mathcal{J}_q,$ and by \cref{t_0'}, some  $0< t_0/\beta< q$ such that 
    
    $$a(\beta t_0/\beta) u(\varphi(x))\Z^{n+1}\notin K_{\xi},$$
    which is a contradiction to \cref{bounded}.
\end{proof}
\begin{proposition}
\label{result}
    For $\forall q\in\N, q>\lambda_1, \forall~ 1\leq Q<e^{\beta(q-\lambda_1)},\forall I\in\mathcal{J}_q,\forall x\in I$, there are no non-zero integer solutions $(m,p_1,...,p_n)$ to the system
    \begin{equation}
        \label{simul system}
        |m|\leq Q, |m\varphi_i(x)-p_i|<\frac{k_1}{n}Q^{-r_i},~\forall~ 1\leq i\leq n.
    \end{equation}
\end{proposition}

\begin{proof}
   Suppose by contradiction, $(m,p_1,\cdots,p_n)$ be a non-zero integer solution to \cref{simul system} with a positive number $1\leq Q<e^{\beta(q-\lambda_1)}$ and $x\in I$ for some $I\in\mathcal{J}_q$. Let $$L_0(m,p_1,\cdots,p_n):=m,L_i(m,p_1,\cdots,p_n):=m\varphi_i(x)-p_i,~\forall~ 1\leq i\leq n.$$ The transposed system is $L_0'(v_0,v_1,...,v_n)=v_0+\sum_{i=1}^n\varphi_i(x)v_i$ and $L_i'(v_0,v_1,...,v_n)=-v_i,~\forall~ 1\leq i\leq n$. Then \cref{simul system} implies \cref{system1} with these $L_0,L_1,...,L_n$ and $T_0=Q,T_i=\frac{k_1}{n}Q^{-r_i},~\forall 1\leq i\leq n$. By \cref{transfer}, we have a non-zero integer solution $(v_0,v_1,\cdots,v_n)$ to \cref{system2} with these $L_i', ~0\leq i\leq n$ and $\iota=\frac{k_1}{n}$. 
   So $(v_0,v_1,...,v_n)$ is a non-zero integer solution to  \cref{dual system} with $H=Q$, which is a contradiction to \cref{dual}.
\end{proof}

Define a new quantity $c$,
\begin{equation}
\label{def: c}
    c:=\frac{k_1^{1+r_1}d_1}{200 n(f_0+\max_{2\leq i\leq n}d_i)(1+d_1)^2R^4},
\end{equation}
where $d_i$ is as in \eqref{eqn:d_i}, and \begin{equation}\label{def: f_0}
    f_0:=1+\sup_{x\in I_0,  i=1,\cdots, n}|\varphi_i'(x)|.
\end{equation}
We take \begin{equation}\label{R_0'} R>\frac{1}{|I_0|}.\end{equation} 

Given $m\in\N$, let us partition $I_0$ as a disjoint union of $m^{\star}:=\lceil\frac{2d_1}{c}m^{r_1}\rceil\stackrel{\eqref{def: c}}{\geq} 2$ many intervals such that for $1\leq j\leq m^\star-1,$ $\vert I_0^{j,m}\vert=\frac{c|I_0|}{2d_1 m^{r_1}}$, and $\vert I_0^{m^\star, m}\vert\leq \frac{c|I_0|}{2d_1 m^{r_1}}.$

Suppose the center of each $I_0^{j,m}$ is $x_0^{j,m}.$
For $m\in\N, 1\leq j\leq m^\star,$ let us denote $\theta_1^{j,m}:= \theta_1(x_0^{j,m}),$ and let 
\begin{equation}
\label{def: Vj}
    \mathcal{V}^{j,m}:=\left\{\frac{\mathbf{p}}{m} ~\left|~\begin{aligned}&\frac{p_1+\theta_1^{j,m}}{m}\in 3^{n+1}I_0,~ \bp\in\Z^n\text{ and }\\&\left|\varphi_i\left(\frac{p_1+\theta_1^{j,m}}{m}\right)-\frac{p_i+\theta_i\left(\frac{p_1+\theta_1^{j,m}}{m}\right)}{m}\right|<\frac{(f_0+d_i) c}{m^{1+r_i}}, \forall~ 2\leq i\leq n\end{aligned}\right.\right\}, 
\end{equation} 
and \begin{equation}\label{def V upper union}\mathcal{V}:=\bigcup_{m\in\N,m>17|I_0|^{-1}d_1}\bigcup_{j=1}^{m^\star} \mathcal{V}^{j,m}.\end{equation}
For any $\bv=\frac{\bp}{m}$, $1\leq k\leq m^\star,$ let \begin{equation}\label{def: inho_interval1}
    \Delta_{\ta}^{k,m}(\bv):=\left\{x\in I^{k,m}_0:\left|x-\frac{p_1+\theta_1(x)}{m}\right|<\frac{c}{m^{1+r_1}}\right\}. 
    \end{equation}
    
    Note that $\{\Delta_{\ta}^{k,m}(\bv)\}_{k=1}^{m^\star}$ is a collection of disjoint domains. For $\bv=\frac{\mathbf{p}}{m}\in \Q^n, 1\leq j\leq m^\star$, let

\begin{equation}
\label{def: inho_interval}
    \widetilde\Delta^{j,m}_{\ta}(\bv):=\left\{x\in I^{j,m}_0:\left|x-\frac{p_1+\theta^{j,m}_1}{m}\right|<\frac{c}{2m^{1+r_1}}\right\}.
\end{equation}
We have the following lemma showing that these intervals are comparable with the domains  
 $\Delta^{j,m}_{\ta}(\bv)$.
\begin{lemma}\label{interval inside domain}For every $m\in \N,\bp\in\Z^n$, $1\leq j\leq m^\star,$ $\bv=\frac{\bp}{m},$
$$\widetilde\Delta^{j,m}_{\ta}(\bv)\subseteq \Delta^{j,m}_{\ta}(\bv)\subseteq 4\widetilde\Delta^{j,m}_{\ta}(\bv).$$
\end{lemma}
\begin{proof}

Pick any $x\in \widetilde\Delta^{j,m}_{\ta}(\bv).$ Since $x\in I^{j,m}_0,$ and $\vert I_0\vert<1$ by \eqref{eqn: lengthI_0}, $\vert x-x_0^{j,m}\vert< \frac{c}{4d_1 m^{r_1}}$.
Then $$\begin{aligned}
   \left|x-\frac{p_1+\theta_1(x)}{m}\right|&< \left|x-\frac{p_1+\theta_1^{j,m}}{m}\right|+ \left\vert \frac{\theta_1(x)-\theta_1^{j,m}}{m}\right\vert\\
   &< \frac{c}{2 m^{1+r_1}}+ d_1 \frac{c}{4d_1 m^{1+r_1}}< \frac{c}{m^{1+r_1}}.
\end{aligned}$$
Now, suppose $x\in \Delta^{j,m}_{\ta}(\bv),$
$$\begin{aligned}
   \left|x-\frac{p_1+\theta_1^{j,m}}{m}\right|&< \left|x-\frac{p_1+\theta_1(x)}{m}\right|+ \left\vert \frac{\theta_1(x)-\theta_1^{j,m}}{m}\right\vert\\
   &< \frac{c}{ m^{1+r_1}}+ d_1 \frac{c}{4d_1 m^{1+r_1}}< \frac{2c}{m^{1+r_1}}.
\end{aligned}$$

\end{proof}

Next, we have the following lemma that shows that $\Delta^{j,m}_{\ta}(\bv)$ are the \textit{dangerous domains} to avoid.
\begin{lemma}\label{avoiding inho_interval}
Let $I_0$, $\mathcal{V}^{j,m}$ and $\Delta^{j,m}_{\ta}(\bv)$ be as before. Then
\begin{equation*}
I_0\setminus\bigcup_{m> 17 \vert I_0\vert^{-1} d_1}\bigcup_{j=1}^{m^\star}\bigcup_{\bv\in\mathcal{V}^{j,m}}\Delta^{j,m}_{\ta}(\bv)\subseteq\varphi^{-1}(\mathbf{Bad}_{\ta}(\br)).
\end{equation*}
\end{lemma}
\begin{proof}
Suppose $x\notin\varphi^{-1}(\mathbf{Bad}_{\ta}(\br))$ and $x\in I_0$, then there exists $(m,\bp)\in\N\times \Z^n$, with
\begin{equation}
\label{large m}
    m> 17|I_0|^{-1}d_1,
\end{equation}
such that
\begin{equation}\label{phi1}
\left|x-\frac{p_1+\theta_1(x)}{m}\right|\stackrel{\eqref{Change of variable}}{=}\left|\varphi_1(x)-\frac{p_1+\theta_1(x)}{m}\right|<\frac{c}{4m^{1+r_1}},
\end{equation}
and 
\begin{equation}
\left|\varphi_i(x)-\frac{p_i+\theta_i(x)}{m}\right|<\frac{c}{4m^{1+r_i}},\forall i\geq 2.
\label{phii}
\end{equation}
Since $x\in I_0$, there exists $1\leq j\leq m^\star$ such that $x\in I_0^{j,m}.$
Moreover, \begin{equation}\label{eqn: constant theta near theta func}\begin{aligned}\left|x-\frac{p_1+\theta^{j,m}_1}{m}\right|&< \left\vert x-\frac{p_1+\theta_1(x)}{m}\right\vert + \left\vert\frac{\theta_1(x)-\theta_1^{j,m}}{m}\right\vert\\
 & \stackrel{\eqref{phi1}}{<}\frac{c}{4m^{1+r_1}}+ \frac{d_1}{m} \vert x- x_0^{j,m}\vert\\
 & <\frac{c}{4m^{1+r_1}}+ d_1 \frac{c}{4d_1 m^{1+r_1}}=  \frac{c}{2m^{1+r_1}}.\end{aligned}\end{equation}
On combining with \cref{large m}, \cref{def: c}, \cref{R_0'} and $x\in I_0$, \cref{eqn: constant theta near theta func} implies $\frac{p_1+\theta^{j,m}_1}{m}\in 3^{n+1}I_0. $
For $\forall~ i\geq 2,$
\begin{equation*}
\begin{aligned}
   &\left\vert\varphi_i\left(\frac{p_1+\theta^{j,m}_1}{m}\right)-\frac{p_i+\theta_i\left(\frac{p_1+\theta^{j,m}_1}{m}\right)}{m}\right\vert\\
    & \leq\left|\varphi_i\left(\frac{p_1+\theta^{j,m}_1}{m}\right)-\varphi_i(x)\right|+\left|\varphi_i(x)-\frac{p_i+\theta_i\left(\frac{p_1+\theta^{j,m}_1}{m}\right)}{m}\right|\\
    &\leq\left|\varphi_i\left(\frac{p_1+\theta^{j,m}_1}{m}\right)-\varphi_i(x)\right|+\left|\varphi_i(x)-\frac{p_i+\theta_i(x)}{m}\right|+\\& \hspace{2.2 cm}\frac{1}{m} \left\vert \theta_i(x)-\theta_i\left(\frac{p_1+\theta^{j,m}_1}{m}\right)\right\vert.
\end{aligned}\end{equation*}
Therefore, by \cref{phii} and mean value theorem,
\begin{equation}\label{eqn: second inequality}
\begin{aligned}
   &\left\vert\varphi_i\left(\frac{p_1+\theta^{j,m}_1}{m}\right)-\frac{p_i+\theta_i\left(\frac{p_1+\theta^{j,m}_1}{m}\right)}{m}\right\vert\\
    &<(f_0-1)\left|x-\frac{p_1+\theta^{j,m}_1}{m}\right|+\frac{c}{4 m^{1+r_i}}+ \\ & \hspace{ 2.2 cm}\frac{d_i}{m} \left\vert x-\frac{p_1+\theta_1^{j,m}}{m}\right\vert \\
    &\stackrel{\eqref{eqn: constant theta near theta func}}{<}(f_0-1)\frac{c}{m^{1+r_1}}+\frac{c}{m^{1+r_i}}+d_i \frac{c}{m^{1+r_1}}\\
    &\stackrel{\eqref{weights chain}}{\leq}\frac{(f_0+d_i)c}{m^{1+r_i}}.
\end{aligned}\end{equation}
By \cref{eqn: second inequality}, $\bv:=\frac{\mathbf{p}}{m}\in\mathcal{V}^{j,m}$ and by \cref{eqn: constant theta near theta func}, $x\in\widetilde\Delta^{j,m}_{\ta}(\bv)$. By \cref{interval inside domain}, $x\in\Delta^{j,m}_{\ta}(\bv)$ which completes the proof.
\end{proof}

For $q\in\N$, we define 
\begin{equation}
\label{def: V_n}
    \mathcal{V}_q:=\left\{\bv=\frac{\mathbf{p}}{m}\in\mathcal{V}, 2c|I_0|^{-1}R^{q+1}\leq m^{1+r_1}<2c|I_0|^{-1}R^{q+2}\right\},
\end{equation} and $\mathcal{V}_0:=\emptyset.$ 
Note that 
\begin{equation}\mathcal{V}_q=\emptyset, ~~\forall~~ 0\leq q\leq \lambda_1+1, q\in\N.\label{empty}\end{equation} This is because if $\frac{\bp}{m}\in\mathcal{V}_q,$ \begin{equation}\label{upper bound on m}
1\leq m\stackrel{\eqref{R_0'}, \eqref{def: c}}{<}\left(\frac{k_1^{1+r_1}}{n}R^{q-1}\right)^{\frac{1}{1+r_1}}<e^{\beta (q-1-\lambda_1)}.\end{equation} 
\noindent By \cref{beta}, the last inequality holds if $\frac{k_1^{1+r_1}}{n } < R^{-\lambda_1}$, which is true by \cref{upper bound on k_1r_1}. Now \cref{upper bound on m} implies for $q\leq \lambda_1+1$, there is no such $m$, hence $\mathcal{V}_q=\emptyset.$ 
Then by definition, we have
\begin{equation}
    \bigcup_{q\in\N }\mathcal{V}_q=\mathcal{V}.
    \label{eqn: union}
\end{equation}
Also, let us denote 
$\mathcal{V}_q^{j,m}:=\mathcal{V}_q\cap \mathcal{V}^{j,m}.$ Then by \cref{eqn: union}, \begin{equation}\label{union_sublevel}\bigcup_{q\in\N} \mathcal{V}_q^{j,m}=\mathcal{V}^{j,m}.\end{equation}
Then we have the following two propositions, which are crucial to constructing Cantor set later.

\begin{proposition}

Suppose $\bv=\frac{\bp}{m}\in\mathcal{V}$. Then, there are at most two $1\leq j\leq m^\star$ such that $\Delta^{j,m}_{\ta}(\bv)\neq\emptyset$.

\label{key1}
\end{proposition}

\begin{proof}

Suppose $1\leq j_1\neq j_2\neq \cdots \neq j_k\leq m^\star$ such that for $1\leq s \leq k$, $\Delta_{\ta}^{j_s,m}(\bv)\neq\emptyset.$ Specifically,
\begin{equation}
\label{inho_10}
    \forall 1\leq s \leq k,~ \exists~ x_{j_s}\in I^{j_s,m}_0 
    \textit{such that }\left|x_{j_s}-\frac{p_{1}+\theta_{1}(x_{j_s})}{m}\right|<\frac{c}{m^{1+r_1}}.
\end{equation}
Hence for any $1\leq i,i'\leq k$, $$\begin{aligned}\vert x_{j_i}-x_{j_{i'}}\vert &\leq \sum_{s=i,i'}\left\vert x_{j_s}-\frac{p_1+\theta_1(x_{j_s})}{m}\right\vert +\frac{1}{m}\vert \theta_1(x_{j_i})-\theta_1(x_{j_{i'}})\vert\\
 & \leq \frac{2c}{m^{1+r_1}}+ \frac{d_1}{m}\vert x_{j_i}-x_{j_{i'}}\vert.
 \end{aligned}$$ Since $m\in \mathcal{V},$ $m-d_1>0.$ Hence we  have \begin{equation}\label{eqn99}
     \vert x_{j_i}-x_{j_{i'}}\vert \leq \frac{2c}{m^{r_1}(m-d_1)}.
 \end{equation}  Moreover, by $|I_0|<1$ 
 and \cref{def V upper union}, we have $m-d_1>16|I_0|^{-1}d_1$. Thus, by \eqref{eqn99},
  $\vert x_{j_i}-x_{j_{i'}}\vert < \frac{c|I_0|}{8d_1 m^{r_1}}.$ 
 This together with the fact that for any $1\leq s\leq k$, $x_{j_s}\in I_0^{j_s,m}$, implies that there are at most two such $s$ satisfying \cref{inho_10}. 
\end{proof}

\begin{proposition}
$\forall q\in\N,~\forall I\in\mathcal{J}_{q-1}$, there will be at most one $\bv=\frac{\bp}{m}\in\mathcal{V}_q$ such that for some $1\leq j\leq m^\star$, $\bv\in\mathcal{V}^{j,m}$ and $\Delta^{j,m}_{\ta}(\bv)\cap I\neq\emptyset$.
\label{key0}
\end{proposition}

\begin{proof}
Since by \cref{empty}, $\mathcal{V}_q=\emptyset$, for $q\leq \lambda_1+1$, the proposition is vacuously true in this range of $q$. Thus, we prove the proposition for $q\geq \lambda_1+2$. Suppose, by contradiction, there are $q\geq \lambda_1+2$, $I\in\mathcal{J}_{q-1}$, and two such distinct $\bv_s=\frac{\bp_s}{m_s}\in\mathcal{V}_q$, $\bp_{s}=(p_{i,s})_{i=1}^n, s=1,2$, for some $1\leq j_s\leq m_s^\star,$ $\bv_s\in\mathcal{V}^{j_s,m_s}$ and there exists $x_s\in\Delta_{\ta}^{j_s,m_s}(\bv_s)\cap I\neq\emptyset$. 
In particular,
\begin{equation}
\label{inho_1}
    \forall s=1,2,~ \exists~ x_s\in I\cap I^{j_s,m_s}_0, \textit{such that }\left|x_s-\frac{p_{1,s}+\theta_{1}(x_s)}{m_s}\right|<\frac{c}{m_s^{1+r_1}},
\end{equation}
and 
\begin{equation}
\label{inho_2}
    \forall i\geq 2,~ \left|\varphi_i\left(\frac{p_{1,s}+\theta^{j_s,m_s}_1}{m_s}\right)-\frac{p_{i,s}+\theta_i\left(\frac{p_{1,s}+\theta^{j_s,m_s}_1}{m_s}\right)}{m_s}\right|<\frac{(f_0+d_i)c}{m_s^{1+r_i}}.
\end{equation}
Without loss of generality,
\begin{equation}
    m_1\geq m_2. 
    \label{assume}
\end{equation}
Since $x_1,x_2\in I$, we have
\begin{equation}
    |x_1-x_2|<|I|=|I_0|R^{-q+1}.
    \label{diam}
\end{equation}
Since $\theta_1$ is Lipschitz, and $x_s\in I_0^{j_s,m_s}$, by \cref{inho_1},  \begin{equation}\label{eqn: shifted rational near xs }
\forall ~s=1,2,~\left|x_s-\frac{p_{1,s}+\theta_{1}^{j_s,m_s}}{m_s}\right|<\frac{5c}{4m_s^{1+r_1}}.
\end{equation}
Then we have,
\begin{equation}\label{ho1}
\begin{aligned}
    &|(m_1-m_2)x_1-(p_{1,1}-p_{1,2})|\\
    &\leq m_2|x_1-x_2|+\sum_{s=1,2}m_s\left |x_s-\frac{p_{1,s}+\theta_1(x_s)}{m_s}\right|+\vert \theta_1(x_1)-\theta_1(x_2)\vert \\
    &\stackrel{\eqref{diam},\eqref{inho_1}}{<}m_2|I_0|R^{-q+1}+\sum_{s=1,2}\frac{c}{m_s^{r_1}}+ d_1 |I_0|R^{-q+1}\\
    &\stackrel{\eqref{assume}}{\leq}m_2(1+d_1)|I_0|R^{-q+1}+\frac{2c}{m_2^{r_1}}.
\end{aligned}
\end{equation}
Note that since $\frac{\bp_2}{m_2}\in\mathcal{V}_q$, $m_2(1+d_1)|I_0|R^{-q+1}< 2c(1+d_1)R^{3}\stackrel{\eqref{def: c}}{<}\frac{1}{2},$ and $\frac{2c}{m_2^{r_1}}<\frac{1}{2},$ this makes the right side of \eqref{ho1} $<1.$ 

\medskip
Moreover, for $\forall i\geq 2,$ 
we can write $\left|(m_1-m_2)\varphi_i\left(\frac{p_{1,1}+\theta^{j_1,m_1}_1}{m_1}\right) - \left.(p_{i,1}-p_{i,2})\right.\right|$
\begin{align}
    & \leq  m_2 \left|\varphi_i\left(\frac{p_{1,1}+\theta^{j_1,m_1}_1}{m_1}\right)-\varphi_i\left(\frac{p_{1,2}+\theta^{j_2,m_2}_1}{m_2}\right)\right|\nonumber\\ &+\left|\theta_i\left(\frac{p_{1,1}+\theta^{j_1,m_1}_1}{m_1}\right)-\theta_i\left(\frac{p_{1,2}+\theta^{j_2,m_2}_1}{m_2}\right)\right|\nonumber\\ &+\sum_{s=1,2}m_s\left|\varphi_i\left(\frac{p_{1,s}+\theta^{j_s,m_s}_1}{m_s}\right)-\frac{p_{i,s}+\theta_i\left(\frac{p_{1,s}+\theta^{j_s,m_s}_1}{m_s}\right)}{m_s}\right|\nonumber\\\stackrel{MVT, \eqref{inho_2}}{<}& m_2(f_0+d_i)\left(|x_1-x_2|+\sum_{s=1,2}\left|x_s-\frac{p_{1,s}+\theta^{j_s,m_s}_1}{m_s}\right|\right)\nonumber\\&+\sum_{s=1,2}\frac{(f_0+d_i)c}{m_s^{r_i}}\nonumber\\
    \stackrel{\eqref{eqn: shifted rational near xs }, \eqref{diam}}{\leq}&\frac{5(f_0+d_i)c}{m_2^{r_i}}+m_2(f_0+d_i)|I_0|R^{-q+1}\nonumber.
\end{align}
Thus, we get
\begin{equation}\label{ho2}
    \begin{aligned}&\left|(m_1-m_2)\varphi_i\left(\frac{p_{1,1}+\theta^{j_1,m_1}_1}{m_1}\right) - \left.(p_{i,1}-p_{i,2})\right.\right|\\& \leq\frac{5(f_0+d_i)c}{m_2^{r_i}}+m_2(f_0+d_i)|I_0|R^{-q+1}.
\end{aligned}\end{equation}
Using \cref{def: c}, we derive that the right side of \eqref{ho2} is $<1.$
Now note, we may assume $m_1>m_2$. Otherwise, as the right sides of \cref{ho1}, and \cref{ho2} are $<1$, $\bp_1=\bp_2$, giving $\bv_1=\bv_2,$ that contradicts our assumption.
Next, we claim that 
\begin{equation}\label{sol1}
  (m_1-m_2)^{r_1}|(m_1-m_2)x_1-(p_{1,1}-p_{1,2})|<\frac{k_1}{n},
\end{equation}
and for $2\leq i\leq n,$
\begin{equation}\label{sol2}
(m_1-m_2)^{r_i}|(m_1-m_2)\varphi_i(x_1)-(p_{i,1}-p_{i,2})|<\frac{k_1}{n}.
\end{equation}
First, we show that \cref{sol1} is true,
\begin{equation*}
\begin{aligned}
(m_1-m_2)^{r_1}|(m_1-m_2)x_1-(p_{1,1}-p_{1,2})|&\stackrel{\eqref{ho1},\eqref{assume}}{<}m_1^{1+r_1}(1+d_1)|I_0|R^{-q+1}+\frac{2cm_1^{r_1}}{m_2^{r_1}}\\&\stackrel{\eqref{def: V_n}}{<}2c(1+d_1)R^3+2cR\stackrel{\eqref{def: c}}{<}\frac{k_1}{n}.
\end{aligned}
\label{abc1}
\end{equation*}
Next, we show that \cref{sol2} holds.
\begin{equation*}
\label{abc}
    \begin{aligned}
        &(m_1-m_2)^{r_i}|(m_1-m_2)\varphi_i(x_1)-(p_{i,1}-p_{i,2})|\\
        &\leq (m_1-m_2)^{1+r_1}\left|\varphi_i(x_1)-\varphi_i\left(\frac{p_{1,1}+\theta^{j_1,m_1}_1}{m_1}\right)\right|\\  & \hspace{2 cm}+(m_1-m_2)^{r_i}\left|(m_1-m_2)\varphi_i\left(\frac{p_{1,1}+\theta^{j_1,m_1}_1}{m_1}\right)-(p_{i,1}-p_{i,2})\right|\\
        &\stackrel{\eqref{assume},\eqref{ho2},\eqref{eqn: shifted rational near xs },MVT}{\leq}\frac{5 f_0 m_1^{1+r_1} c}{4m_1^{1+r_1}}+\frac{ 5(f_0+d_i)cm_1^{r_i}}{m_2^{r_i}}+ m_1^{1+r_i}(f_0+d_i)|I_0|R^{-q+1}
        \\
        & \hspace{1.5 cm}= \frac{5}{4}f_0c + \frac{ 5 (f_0+d_i)cm_1^{r_i}}{m_2^{r_i}}+ m_1^{1+r_i}(f_0+d_i)|I_0|R^{-q+1}\\
        &\hspace{1.5 cm} \stackrel{\eqref{assume}}{\leq}\frac{10(f_0+d_i)cm_1^{r_i}}{m_2^{r_i}}+m_1^{1+r_i}(f_0+d_i)|I_0|R^{-q+1}
        \\& \hspace{1.2 cm}\stackrel{\eqref{def: V_n},\eqref{weights chain}}{<}10(f_0+d_i)cR+2(f_0+d_i)cR^3\stackrel{\eqref{def: c}}{<}\frac{k_1}{n}.
    \end{aligned}
\end{equation*}
We also have,
\begin{equation*}1\leq m_1-m_2\leq m_1\stackrel{\eqref{upper bound on m}}{<}e^{\beta (q-1-\lambda_1)}.\end{equation*} 
Since $m_1>m_2$, by \cref{Change of variable}, \cref{sol1} and \cref{sol2} the integer vector $(m_1-m_2,\bp_1-\bp_2)$ is a non-zero solution to the system of equations \cref{simul system} with $Q=m_1-m_2<e^{\beta(q-1-\lambda_1)}$, $q\geq\lambda_1+2, \text{ and }I\in\mathcal{J}_{q-1}, x\in I$, which leads to a contradiction to \cref{result}.  

\end{proof}
Combining \cref{key1} and \cref{key0}, we have the following key proposition.
\begin{proposition}
$\forall q\in\N,~\forall I\in\mathcal{J}_{q-1}$, there will be at most one $\bv=\frac{\bp}{m}\in\mathcal{V}_q$ and at most two $1\leq j\leq m^\star$ such that $\bv\in\mathcal{V}^{j,m}$ and $\Delta^{j,m}_{\ta}(\bv)\cap I\neq\emptyset$. 
\label{key}
\end{proposition}

In what follows, we take $p,q\in\N\cup\{0\}.$
For any $q\geq 0$, we define,

\begin{equation}
\label{def: M}
    \mathcal{M}_{q,q+1}:=\left\{I\in\mathcal{I}_{q+2}:\exists~\bv=\frac{\bp}{m}\in\mathcal{V}_{q+1}, \left|  \begin{aligned} &~\bv\in\mathcal{V}^{j,m} \text{ and }\\ &~\Delta_{\ta}^{j,m}(\bv)\cap I\neq\emptyset\end{aligned} \right. \text{ for some } 1\leq j\leq m^\star\right\}.
\end{equation}
and let $\mathcal{M}_{p,q}:=\emptyset, 0\leq p\leq q, p\neq q-1$.
Next, for $0\leq p\leq q,$ we define 
$$f_{p,q}:=\max_{J\in\mathcal{J}_{p}}\#\{I\in\mathcal{M}_{p,q}: I\subset J\}.$$
Trivially,
\begin{equation}
    \label{f0q}
    f_{q,q}=0.
\end{equation}
And for $p<q,$ we have the following proposition.
\begin{proposition}
\label{fpq}
    For all $p\geq 0,q\geq 1$ with $p< q$, we have $$f_{p,q}\leq 6 .$$ 
\end{proposition}
\begin{proof}
    If $p\neq q-1$, then $f_{p,q}=0$, which trivially satisfies the upper bound.

    If $p=q-1$, then given $J\in\mathcal{J}_p$, by \cref{key}, there is at most one $\bv=\frac{\bp}{m}\in\mathcal{V}_q$ and at most two $j\in[1,m^*]\cap\N$ such that $\bv\in \mathcal{V}^{j,m}$, and \begin{equation}\label{split in jm} \Delta^{j,m}_{\ta}(\bv)\cap J\neq\emptyset. \end{equation}  By \cref{interval inside domain}, we have that

    $$4\widetilde \Delta^{j,m}_{\ta}(\bv)\cap J\neq\emptyset.$$
    Also, by \cref{def: V_n} and \cref{def: inho_interval}, $|4 \widetilde\Delta^{j,m}_{\ta}(\bv)|\leq 2|I_0|R^{-q-1}$. 
    Then for each $j$, the domain $\Delta^{j,m}_{\ta}(\bv)$ can intersect at most $3$ of intervals in $\{I\subset J:I\in\mathcal{I}_{q+1}\}$, which completes the proof.
\end{proof}
\subsection{The new Cantor set}.

We construct the Cantor set by induction on $q\geq 0$. 
Let $\mathcal{J}_0'=\{I_0\}$. Let us assume that $\mathcal{J}'_1,\cdots, \mathcal{J}'_q$, for $q\geq 0$ are already constructed. Let $\mathcal{I}_{q+1}'=\mathbf{Par}_R(\mathcal{J}_q').$ Then we define,
$$\hat{\mathcal{J}}_{p,q}':=\mathcal{I}_{q+1}'\cap(\mathcal{M}_{p,q}\cup\hat{\mathcal{J}}_{p,q}),~~\forall~~ 0\leq p\leq q .$$ Now, let $\hat{\mathcal{J}}_{q}'=\cup_{p=0}^q\hat{\mathcal{J}}_{p,q}'$ . Then we remove the subcollection $\hat{\mathcal{J}}_q'$ from $\mathcal{I}_{q+1}'$, i.e. $\mathcal{J}_{q+1}'=\mathcal{I}_{q+1}'\setminus\hat{\mathcal{J}}_q'.$ 
\begin{lemma}\label{prime inside old}
For $q\geq 1$, $\mathcal{J}_q'\subseteq \mathcal{J}_q$, and $\mathcal{I}'_{q}\subseteq\mathcal{I}_q$. 
\end{lemma}
\begin{proof}
 We prove this lemma by induction. When $q=1$, note that $\mathcal{I}_1'=\mathcal{I}_1,$ and $\mathcal{\hat J}'_0= \mathcal{\hat J}'_{0,0}= \mathcal{I}_1'\cap \mathcal{\hat J}_{0,0}= \mathcal{I}_1\cap \mathcal{\hat J}_{0}= \mathcal{\hat J}_{0}.$ So $\mathcal{J}_1'=\mathcal{I}_1'\setminus \mathcal{\hat J}_{0}'= \mathcal{I}_1\setminus \mathcal{\hat J}_{0}=\mathcal{J}_1$.

 By the induction hypothesis, $\mathcal{J}'_i\subseteq \mathcal{J}_i$ and $\mathcal{I}'_i\subseteq \mathcal{I}_i$ for $i=1,\cdots,q.$ Thus, $\mathcal{I}'_{q+1}= \mathbf{Par}_R(\mathcal{J}_q')\subseteq \mathbf{Par}_R(\mathcal{J}_q)= \mathcal{I}_{q+1}.$ Note that $\hat{\mathcal{J}}_{p,q}\cap \mathcal{I}_{q+1}'\subseteq \hat{\mathcal{J}}_{p,q}',\forall~0\leq p\leq q\implies \hat{\mathcal{J}}_{q}\cap \mathcal{I}_{q+1}'\subseteq \hat{\mathcal{J}}_{q}'.$ Then $\mathcal{J}_{q+1}' \subseteq \mathcal{I}_{q+1}'\setminus \left( \mathcal{I}_{q+1}'\cap \hat{\mathcal{J}}_q\right)= \left(\mathcal{I}_{q+1}'\cap\mathcal{I}_{q+1}\right)\setminus \left( \mathcal{I}_{q+1}'\cap \hat{\mathcal{J}}_q\right)= \left(\mathcal{I}_{q+1}\setminus\hat{\mathcal{J}}_q\right)\cap\mathcal{I}_{q+1}'\subseteq\mathcal{J}_{q+1}$ and the lemma follows.
\end{proof}
Lastly, we define a new Cantor set as follows:

\begin{equation}
    \label{new cantor}
\mathcal{K}_{\infty}':=\bigcap_{q\geq 0}\bigcup_{I\in\mathcal{J}_q'}I.
\end{equation}

\begin{lemma}\label{new inside old} Let $\mathcal{K}'_{\infty}$ be the Cantor set we constructed in \cref{new cantor}. 
Then,
$$\mathcal{K}'_{\infty}\subseteq \mathcal{K}_{\infty}\subseteq \supp{\mu},$$ where $\mathcal{K}_{\infty}$ is the Cantor set defined in \cref{subsection: Notation}
, and $\mu$ is as in \cref{thm: twin main}
\end{lemma}
\begin{proof}

    Note $\mathcal{J}_0'=\mathcal{J}_0$, and by \cref{prime inside old}, $\mathcal{J}'_{q}\subseteq\mathcal{J}_q,\forall q\geq 1$. Hence $\mathcal{K}_{\infty}'=\cap_{q\geq 0}\cup_{I\in\mathcal{J}_q'}I\subseteq\cap_{q\geq 0}\cup_{I\in\mathcal{J}_q}I=\mathcal{K}_{\infty}.$ 
    The last inclusion, $\mathcal{K}_{\infty}\subseteq \supp{\mu}$, was already proved in \cite[Proposition 2.6]{BNY22}.
\end{proof}
\begin{lemma}\label{Cantor inside bad}
Let $\mathcal{K}'_{\infty}$ be as before, $\mathbf{Bad}_{\ta}(\mathbf{r})$ be as defined in \cref{def: simul_inho}. Then
$$\mathcal{K}'_{\infty}\subseteq\varphi^{-1}(\mathbf{Bad}_{\ta}(\mathbf{r})).$$\end{lemma}

\begin{proof}  
In the construction of the Cantor set $\mathcal{K}'_{\infty}$, we have that all the intervals in $\mathcal{J}_q'$ avoids $\mathcal{M}_{p,q-1}$, and in particular, $\mathcal{M}_{q-2,q-1}.$ when $q\geq2.$ Now in view of \cref{def: M}, we have $\forall q\geq 2, ~\forall~ I\in\mathcal{J}_q',~ \forall~ \bv=\frac{\bp}{m}\in\mathcal{V}_{q-1}$, for all $1\leq j\leq m^\star,$
$$ \Delta_{\ta}^{j,m}(\bv)\cap I=\emptyset \text{ when } \bv\in \mathcal{V}^{j,m}.$$ 
Hence for any $q\geq 2,$

$$ \bigcup_{I\in\mathcal{J}_q'}I\subseteq I_0\setminus 
\bigcup_{m> 17\vert I_0\vert^{-1}d_1}\bigcup_{j=1}^{m^\star}\bigcup_{\bv\in \mathcal{V}_{q-1}^{j,m}}\Delta^{j,m}_{\ta}(\bv).$$
Therefore, by \cref{avoiding inho_interval} and \cref{union_sublevel}, \cref{new cantor}, $$\mathcal{K}_{\infty}'\subseteq I_0\setminus \bigcup_{m> 17\vert I_0\vert^{-1}d_1}\bigcup_{j=1}^{m^\star}\bigcup_{\bv\in\mathcal{V}^{j,m}}\Delta^{j,m}_{\ta}(\bv)\subseteq\varphi^{-1}(\mathbf{Bad}_{\ta}(\br)).$$

\end{proof}

Now by \cref{new inside old} and \cref{Cantor inside bad}, we have \begin{equation}\label{new cantor set inside curve and support}
\mathcal{K}'_\infty\subseteq\varphi^{-1}(\mathbf{Bad}_{\ta}(\mathbf{r}))\cap\supp\mu.
\end{equation}

\begin{remark}
Note that since  $\mathcal{K}'_\infty\subseteq \mathcal{K}_\infty,$ by \cite[Proposition 2.6]{BNY22}, and \cref{new cantor set inside curve and support},
$$
\mathcal{K}'_\infty\subseteq \varphi^{-1}(\mathbf{Bad}_{\ta}(\mathbf{r}))\cap \varphi^{-1}(\mathbf{Bad}(\mathbf{r}))\cap\supp\mu.
$$
\end{remark}

Now to prove \cref{thm: twin main}, the main task is to show that $\mathcal{K}'_{\infty}$ is non-empty, which will be the content of the rest of this section.

\subsection{New upper bounds}
For all $0\leq p\leq q,$ let us define 
$$
h_{p,q}':=\max_{J\in\mathcal{J}'_{p}}\#\{I\in\hat{\mathcal{J} }'_{p,q}:I\subset J\}.$$
    By definition of $\hat{\mathcal{J}}'_{p,q}$ and \cref{prime inside old}, we have that \begin{equation*}\begin{split}h_{p,q}'\leq\max_{J\in\mathcal{J}_{p}}\#\{I\in\hat{\mathcal{J} }'_{p,q}:I\subset J\}
    &\leq \max_{J\in\mathcal{J}_{p}}\#\{I\in\hat{\mathcal{J }}_{p,q}:I\subset J\}+\max_{J\in\mathcal{J}_{p}}\#\{I\in\mathcal{M}_{p,q}: I\subset J\}\\=&h_{p,q}+f_{p,q}.\end{split}\end{equation*} 
    Therefore,
    \begin{equation}\label{triangle}
   h_{p,q}'\leq h_{p,q}+f_{p,q}.\end{equation}

\begin{proposition}
\label{hqq'}
Let $C,\alpha$ be as in \cref{thm: twin main}. For $\forall q\geq 0$, and $R^{\alpha}\geq 21C^2$
    \begin{equation}
        h_{q,q}'\leq R-(4C)^{-2}R^{\alpha}.
    \end{equation}
\end{proposition}
\begin{proof}
    By \cref{triangle}, \cref{f0q}, and \cref{hqq}, we have $h_{q,q}'\leq h_{q,q}+f_{q,q}\leq  R-(4C)^{-2}R^{\alpha}.$
\end{proof}
\begin{proposition}
\label{hpq'}
    There exist constants $R_0'\geq 1,C_0'>0,\eta_0'>0$, such that $\forall R\geq R_0'$ and $~0\leq p<q$,\begin{equation}
    h_{p,q}'\leq C_0'R^{\alpha(1-\eta_0')(q-p+1)}.
    \end{equation}
    
\end{proposition}
\begin{proof}
   Let $R_0, C_0$ and $\eta_0$ be as in \cref{hpq}.  Choose $R_0',C_0',\eta_0'$ such that $R_0'>\max\{ R_0,\frac{1}{|I_0|}\},C_0'\geq\max\{16,8C_0\}$ and $\eta_0=\eta_0'$. Then by \cref{triangle}, \cref{hpq} and \cref{fpq},
    \begin{equation}
        h_{p,q}'\leq h_{p,q}+f_{p,q}\leq C_0R^{\alpha(1-\eta_0)(q-p+1)}+6\leq C_0'R^{\alpha(1-\eta_0')(q-p+1)}.
    \end{equation}
   
\end{proof}

\subsection{Completing the proof of \cref{thm: twin main}}.

As mentioned earlier, the non-emptiness of $\mathcal{K}'_{\infty}$ implies \cref{thm: twin main}. Given \cref{hqq'} and \cref{hpq'}, the rest of the proof is exactly the same as in \cite[Section 4]{BNY22}, which we rewrite for completion. By \cref{nemp}, it suffices to prove that for $q\geq0$, \begin{equation}t_q':=R-h_{q,q}'-\sum_{j=1}^q\frac{h_{q-j,q}'}{\prod_{i=1}^jt_{q-i}'}\geq(6C)^{-2}R^{\alpha}
\label{non}.\end{equation}

We prove this by induction. When $q=0$, $t_0'=R-h_{0,0}'\geq (4C)^{-2}R^{\alpha}$ by \cref{hqq'}, which satisfies \cref{non}.

Suppose \cref{non} holds for $q<q_1$ where $q_1\geq 1$. Then by \cref{hpq'} and \cref{hqq'}, we have $$t_{q_1}'\geq(4C)^{-2}R^{\alpha}-R^{\alpha}\left(C_0'R^{-\eta_0'\alpha}\sum_{j=1}^{\infty}\left(\frac{36C^2}{R^{\eta_0'\alpha}}\right)^j\right).$$
Thus, we only have to make sure that $R$ is large enough such that $C_0'R^{-\eta_0'\alpha}<\frac{1}{32C^2}$ and $\sum_{j=1}^{\infty}(\frac{36C^2}{R^{\eta_0'\alpha}})^j\leq 1$, while satisfying the lower bounds for $R$ in \cref{R_0'}, \cref{hqq'} and \cref{hpq'}, to make the induction work. Thus the proof is done.

We end this section with two minor remarks.

    \begin{remark}Let $\br$ and $\varphi$ be as in \cref{thm: main}.  Let $\tilde\theta_i:U\to \R$ be Lipschitz function for $1\leq i\leq n$. Then we showed,
    $$\{x\in U:\liminf_{q\in\mathbb{Z}\setminus\{0\}, \vert q\vert \to\infty}\max_{1\leq i\leq n}\vert q\varphi_i(x)-\tilde\theta_i(x)\vert_{\mathbb{Z}}^{1/r_i}   \vert q\vert>0 \}$$ 
    is absolute winning on $U.$ The main difference between this statement and \cref{thm: main} is that, in the case of curves, we can take the inhomogeneous function to be any Lipschitz function, and don't have to restrict to composite function $\theta^\varphi_i$ as in \cref{subsection: Notation}.
    \end{remark}
\begin{remark} \label{remk on defn}  Note that  $\textbf{Bad}(\br)$ is the same as the set 
    $$
    \{\mathbf{x}\in\mathbb{R}^n:\inf_{q\in\mathbb{Z}\setminus\{0\}}\max_{1\leq i\leq n}\vert qx_i\vert_{\mathbb{Z}}^{1/r_i}   \vert q\vert>0 \}.
    $$
    But in the inhomogeneous set-up, one may not replace `$\liminf$' in \cref{def: simul_inho} by $`\inf$'. For instance, if $\theta_i(\bx)=ax_i+b$, for $i=1,\cdots,n,$ 
    where $a\in \Z\setminus\{0\}, b\in\Z$, one can easily see that replacing `$\liminf$' by `$\inf$' will make the set of bad to be empty. However, this kind of function is not interesting since the corresponding inhomogeneous set is actually the same as a homogeneous set. So for the sake of generality, we took $`\liminf$' in the definition \eqref{def: simul_inho}. 
    \end{remark}
\section{Final Remark}\label{final}
In this paper, we consider 
weighted inhomogeneous badly approximable vectors in the simultanmeous setting. One can consider a dual analogue of inhomogeneous badly approximable vectors as follows.
For any map $
\theta:\R^n\to\R,$ we define the following set:
\begin{equation}
\label{def: dual_inho}
    \widetilde{\mathbf{Bad}_{\theta}(\mathbf{r})}:=\{\mathbf{x}\in\mathbb{R}^n:\liminf_{\bq=(q_i)\in\Z^n\setminus\{0\}, \Vert \bq\Vert\to\infty}\vert \bq\cdot \bx-\theta(\bx)\vert_{\mathbb{Z}}  \max_{1\leq i\leq n}\vert q_i\vert^{1/r_i} >0 \}.
\end{equation} 
In \cite{Beresnevich2015}, it was shown that $\widetilde{\mathbf{Bad}_{0}(\mathbf{r})}=\mathbf{Bad}_{\mathbf{0}}(\mathbf{r}),$ using Khintchine’s transference principle. But more generally in the inhomogeneous setting, it is not clear if the dual and simultaneous settings are related. Therefore, it is natural to consider the following problem.
 \begin{problem}
     Let $\varphi:U \subset \R\to \R^n$ be as in \cref{thm: main}, and let $\theta:\R^n\to\R$ has nice properties (one may just take a constant map). Prove that $\varphi^{-1}(\widetilde{\mathbf{Bad}_{\theta}(\mathbf{r})})$ is absolute winning on $U.$
 \end{problem}
Another interesting direction is to find the measure of the weighted inhomogeneous badly vectors inside nondegenerate curves. More precisely, 
\begin{problem}
    Let $\varphi:U\subset \R \to \R^n$ and $\ta$ be as in \cref{thm: main}. Show that $$\lambda(\varphi^{-1}(\mathbf{Bad}_{\ta}(\mathbf{r}))=0,$$ where $\lambda$ is the Lebesgue measure on $\R.$ 
\end{problem}
Our main Theorem \ref{thm: main}, guarantees that the set has full Hausdorff dimension. So, it is interesting to find the Lebesgue measure of this set. One can follow the general framework developed in \cite{BDGW_null} to achieve such results. 

\subsection*{Acknowledgments} We thank the anonymous referee for their very detailed report that helped us clarify many details. We thank the University of Michigan for giving us the opportunity to work on this project. SD thanks Ralf Spatzier for encouraging her in pursuing this project which started in an REU program.   
\bibliographystyle{abbrv}
\bibliography{inhomo_curve.bib}

\end{document}